\documentclass[12pt,a4paper]{amsart}

\newtheorem{theorem}{Theorem}[section]
\newtheorem{corol}[theorem]{Corollary}
\theoremstyle{definition}
\newtheorem*{proc}{Procedure}
\theoremstyle{remark}
\newtheorem*{erem}{Remark}
\numberwithin{equation}{section}

\def\hB{\bar{\mathbf B}}	\def\hA{\bar{\mathbf A}}
\def\sbe{\subseteq}		\def\spe{\supseteq}
\def\sb{\subset}
\def\pa{\mathop\mathrm{par}}
\def\ed{\mathop\mathrm{edim}}
\def\chr{\mathop\mathrm{char}}

\def\al{\alpha}	\def\be{\beta}	\def\th{\theta}	\def\la{\lambda}
\def\gM{\mathfrak m}	\def\fK{\mathbf k}
\def\bA{\mathbf A}	\def\bB{\mathbf B}	\def\bC{\mathbf C}
\def\bD{\mathbf D}	\def\bJ{\mathbf J}	\def\bK{\mathbf K}
\def\bL{\mathbf L}	\def\bS{\mathbf S}	\def\kI{\mathcal I}
\def\8{\infty}		\def\hom{\mathop\mathrm{Hom}\nolimits}
\def\iff{if and only if }

\begin{document}

\title{Cubic rings and their ideals}
\author{Yuriy A. Drozd}
\address{Institute of Mathematics, National Academy of Sciences of Ukraine, Tereschenkivska 3, 01601 Kyiv, Ukraine}
\email{drozd@imath.kiev.ua}
\urladdr{www.imath.kiev.ua/$\sim$drozd}
 \author{Ruslan V. Skuratovskii}
\address{Kyiv National Taras Shevchenko University, Department of Mechanics and Mathematics, Volodymyrska 60, 01601 Kyiv, Ukraine}
\email{ruslcomp@mail.ru}
\begin{abstract}
  We give an explicit description of cubic rings over a discrete valuation ring, as well as a description of all ideals of such rings. 
\end{abstract}

\maketitle

\section*{Introduction}

 Ideals of local rings have been studied by a lot of authors from quite different viewpoints. One of the questions that arise with this respect is on the \emph{number of parameters} $\pa(\bC)$ defining the ideals of such a ring $\bC$ up to isomorphism, especially when it is reduced and of Krull dimension $1$. Certainly, it makes sense if the residue field $\fK$ is infinite. In \cite{dr} it was shown that $\pa(\bC)=0$, i.e. $\bC$ has a finite number of ideals (up to isomorphism), \iff $\bC$ is \emph{Cohen--Macaulay finite}, i.e. has a finite number of indecomposable non-isomorphic Cohen--Macaulay modules (in the $1$-dimensional reduced case they coincide with torsion free modules). Then Schappert \cite{sch} proved that a plane curve singularity has at most $1$-parameter families of ideals \iff it dominates one of the \emph{strictly unimodal} plane curve singularities in the sense of \cite{wa}, or, the same, \emph{unimodal} and \emph{bimodal} plane curve singularities in the sense of \cite{avg}. In \cite{dg} this result was generalized to all curve singularities. Note that this time $\pa(\bC)=1$ does not imply that $\bC$ is \emph{Cohen--Macaulay tame}, i.e. has at most $1$-dimensional families of indecomposable Cohen--Macaulay modules. Tameness means that $\bC$ dominates a singularity of type $T_{pq}$ \cite{dg2}. The case $\pa(\bC)>1$ had not been studied before the second author described the one branch singularities of type $W$ such that $\pa(\bC)\le2$ \cite{sk}.

 In this paper we study the \emph{cubic rings}. We describe all such rings, their ideals and, in particular, establish the value $\pa(\bC)$ for any cubic ring $\bC$. As a consequence, we show that a cubic ring is Gorenstein \iff it is a plane curve singularity (i.e. its embedding dimension equals $2$).
  
\section{Generalities}
\label{s1}

 We denote by $\bD$ a discrete valuation ring with the ring of fractions $\bK$, the maximal ideal $\gM=t\bD$ and the residue field $\fK=\bD/t\bD$. A \emph{cubic ring} over $\bD$ is, by definition, a $\bD$-subalgebra $\bC$ in a $3$-dimensional semisimple $\bK$-algebra $\bL$, which is a free $\bD$-module of rank $3$. We also denote $\bA$ the integral closure of $\bD$ in $\bL$ and always suppose that $\bA$ is finitely generated as $\bC$-module. Equivalent condition (see, for instance, \cite{max}): the $\gM$-adic completion $\hat\bC$ of the ring $\bC$ has no nilpotent elements. It is always the case if the algebra $\bL$ is \emph{separable}, for instance, if $\chr\bK=0$. We also set $\bA_m=t^m\bA+\bD$ and $\bJ_m=t\bA_{m-1}=\mathrm{rad}\,\bA_m\ (m>0)$. 

 In what follows, an \emph{ideal} means a \emph{fractional $\bC$-ideal} in $\bK$, i.e. a finitely generated $\bC$-submodule $M\sbe\bK$ such that $\bK M=\bK$. Then $M$ is a free $\bD$-module of rank $3$. We are going to describe all ideals of cubic rings up to isomorphism. It is known (see, for instance, \cite{fad}) that there is a one-to-one correspondence between $\bC$-ideals and $\hat\bC$-ideals, mapping $M$ to its $\gM$-adic completion. This correspondence \emph{reflects isomorphisms}, i.e. maps non-isomorphic ideals to non-isomorphic. So, in what follows we may (and will) suppose that $\bD$ is \emph{complete} with respect to the $\gM$-adic topology. 

Recall also that the \emph{embedding dimension} $\ed\bC$ of a local noetherian ring $\bC$ with the maximal ideal $\bJ$ and the residue filed $\fK$ is defined as $\dim_\fK\bJ/\bJ^2$. If $\bC$ is of Krull dimension $1$ and $\ed\bC=2$, $\bC$ is called a \emph{plane curve singularity}. In the geometric case, when $\bC$ contains a subfield of representatives of $\fK$, it actually means that there is a plane curve $C$ such that $\bC$ is the completion of the local ring of a singular point $x\in C$. 

 From the general theory of ramification in finite extensions we see that the following cases can happen:

\begin{description}
  \item[One branch, ramified case]
  $\bL$ is a field, the maximal ideal of $\bA$ equals $\tau\bA$, $\bA/\tau\bA\simeq\fK$ and $t\bA=\tau^3\bA$.

\item[One branch, non-ramified case]
  $\bL$ is a field, the maximal ideal of $\bA$ equals $t\bA$ and $\bA/t\bA=\fK[\bar\th]$ is a cubic extension of the field $\fK$, where $\bar\th$ is a root of an irreducible cubic polynomial $f(x)\in\fK[x]$. 
 
\item[Two branches, ramified case]  
 $\bL=\bK_1\times\bK$, where $\bK_1$ is a quadratic extension of $\bK$, $\bA=\bD_1\times\bD$, the maximal ideal of $\bD_1$ is $\tau\bD_1$, $\bD_1/\tau\bD_1\simeq\fK$ and $t\bD_1=\tau^2\bD_1$.

\item[Two branches, non-ramified case]
 $\bL=\bK_1\times\bK$, where $\bK_1$ is a quadratic extension of $\bK$, $\bA=\bD_1\times\bD$, the maximal ideal of $\bD_1$ is $t\bD_1$ and $\bD_1/\tau\bD_1=\fK[\bar\th]$ is a quadratic extension of the field $\fK$, where $\bar\th$ is a root of an irreducible quadratic polynomial $f(x)\in\fK[x]$. 

\item[Three branches case]
 $\bL=\bK^3$, $\bA=\bD^3$.
\end{description}

 We recall \cite{fac,di} that, for any cubic ring $\bC$, every ideal of $\bC$ is isomorphic either to an \emph{over-ring} of $\bC$, i.e. a cubic ring $\bB$ such that $\bC\sbe\bB\sb\bL$, or to the \emph{dual ideal} $\bB^*=\hom_\bD(\bB,\bD)$ of such an over-ring. Hence, to describe all ideals of $\bC$, we only need to describe over-rings of $\bC$. Obviously, any cubic ring in $\bL$ contains some $\bA_m$. Therefore, to describe all cubic rings (so their ideals as well), we have to describe the over-rings of $\bA_m$. If $\bB$ is an over-ring of $\bC$, they also say that $\bB$ \emph{dominates} $\bC$.

 Since the unique (up to isomorphism) $\bA$-ideal is $\bA$ itself, we proceed by induction: supposing that all over-rings of $\bA_m$ are known, we find all over-rings of $\bA_{m+1}$. If $\bC$ is an over-ring of $\bA_{m+1}$, then $\bB=\bC\bA_m$ is an over-ring of $\bA_m$, $t\bB\sb\bC$ and $\bC/t\bB$ is a $\fK$-subalgebra in $\bB/t\bB$. If $\bB\spe\bA_{m-1}$, then $t\bB\spe\bJ_m$, hence, $\bC\spe\bJ_m+\bD=\bA_m$. Therefore, the following procedure gives all over-rings of $\bA_{m+1}$ which are not over-rings of $\bA_m$:

\begin{proc}
\begin{itemize}
\item[]
\item   For every over-ring $\bB$ of $\bA_m$, which is not an over-ring of $\bA_{m-1}$, calculate $\hB=\bB/t\bB$. Set $\hA=(\bA_m+t\bB)/t\bB\sbe\hB$.
\item  Find all proper subalgebras $\bS\sb\hB$ such that $\hA\bS=\hB$. Equivalently, the natural map $\bS\to\bB/\bB\bJ_m$ must be surjective.
\item  For each such $\bS$ take its preimage in $\bB$. 
\end{itemize}
\end{proc}
 
\section{Calculations}

\subsection{One branch, ramified case}
\label{s21}

 We set
\begin{align*}
  \bC_{2r}(\al)&=\bD+t^r\al\bD+t^{2r}\bA, \text{ where } v(\al)=1,\\
  \bC_{2r+1}(\al)&=\bD+t^r\al\bD+t^{2r+1}\bA, \text{ where } v(\al)=2,
\end{align*} 
 where $v$ is the discrete valuation related to the ring $\bA$, i.e. $v(\al)=k$ means that $\al\in\tau^k\bD\setminus\tau^{k+1}\bD$. Note that $\bC_0(\al)=\bA$.  Obviously, $\al$ can be uniquely chosen as $\tau+a\tau^2$ for $\bC_{2r}$ and as $\tau^2+at\tau$ for $\bC_{2r+1}$, where $a\in\bD$ is defined modulo $t^r$.

 \begin{theorem}\label{21}
  Every over-ring of $\bA_m$ coincides with $t^k\bC_r(\al)+\bD$ for some $k,r$ such that $r+k\le  m$ and some $\al$. The rings $\bC_r(\al)$ are just all plane curve singularities in this case.
\end{theorem}
\begin{proof}
For $m=1$ it is easy and known \cite{dr,ja}. So, we use the Procedure for $m>1$, setting $\bB=t^k\bC_r(\al)+\bD$, where $k+r=m$. Then the basis of $\hB$ consists of the classes of the elements $\{1,t^h\al,t^m\tau^s\}$, where $h=k+[r/2]$, $s\in\{1,2\}$ and $s\equiv r\pmod2$. Since $t^h\al\notin\bJ_m$, the subalgebra $\bS$ necessarily contains the class of $t^h\al+c t^m\tau^s$ for some $c\in\bD$. If $k=0$, then $m=r$ and $v(t^m\tau^s)=2v(t^h\al)$. Therefore, $\hB$ has no proper subalgebra containing  the class of $t^h\al+c t^m\tau^s$. If $k>0$, the preimage of $\bS$  is $\bD+(t^h\al+ct^m\tau^s)\bD+t^{m+1}\bA$. It coincides with $t^{k-1}\bC_{r+2}(\al')+\bD$ where $\al'=\al+ct^{m-h}\tau^s$.

Now one easily checks that $\ed\bC_r(\al)=2$, while $\ed\bC=3$ for all other rings. 
\end{proof}

\subsection{One branch, non-ramified case}
\label{s22}

We set $\bC_r(\al)=\bD+t^r\al\bD+t^{2r}\bA_0$, where $\al\in\bA^\times\setminus\bD$. Again $\bC_0(\al)=\bA_0$. Note that $\al$ can be uniquely chosen as $\th+a\th^2$, where $\th$ is a fixed preimage of $\bar\th$ in $\bD_1$ and $a\in\bD$ is defined modulo $t^r$.

\begin{theorem}\label{22}
  Every over-ring of $\bA_m$ coincides with $t^k\bC_r(\al)+\bD$ for some $k,r$ and $\al$ with $2r+k\le m$. The rings $\bC_r(\al)$ are just all plane curve singularities in this case.
\end{theorem}

\begin{proof}
 For $m=1$ it is obvious. So, using the Procedure for $m>1$, we set $\bB=t^k\bC_r(\al)+\bD$ with $2r+k=m$. Then  a basis of $\hB$ consists of the classes of elements $\{1,t^{r+k}\al,t^m\al^2\}$ for some $\al^2\in\bA^\times\setminus(\bD+\al\bD)$. Since $t^{r+k}\al\notin\bJ_m$, $\bS$ must contain the class of an element $t^{r+k}\al'=t^{r+k}\al+c t^m\al^2$ for some $c\in\bD$. As above, it is impossible if $k=0$. If $k>0$, then the preimage of $\bS$ is $\bD+t^{r+k}\al'+t^{m+1}\bA=t^{k-1}\bC_{r+1}(\al')+\bD$.

Now one easily checks that $\ed\bC_r(\al)=2$, while $\ed\bC=3$ for all other rings.
\end{proof}

\subsection{Two branches, ramified case}
\label{s23}

We denote by $v$ the valuation defined by the ring $\bD_1$, by $e$ the idempotent in $\bA$ such that $e\bA=\bD_1$ and set
\begin{align*}
  \bC_{l,q}(\al)&=\bD+t^l(e+t^q\al)\bD+t^r\bA, \text{ where } r=2l+q,\\
  \bC_r(\al)&=\bD+t^r\al\bD+t^{2r+1}\bA.
\end{align*}
 In both cases $\al\in\bD_1$ and $v(\al)=1$, where $v$ is the valuation defined by the ring $\bD_1$. Obviously, $\al$ can be uniquely chosen as $a\tau$, where $a\in\bD$ is defined modulo $r$.  Note that $\bC_{0,q}(\al)=\bD+e_1\bD+t^q\bA$ are just all decomposable rings in this case and $\bC_{0,0}(\al)=\bA$.

\begin{theorem}\label{23}
  Every over-ring of $\bA_m$ coincides with either $t^k\bC_{l,r}(\al)+\bD$ or $t^k\bC_r(\al)+\bD$, where $k+r\le m$. The rings $\bC_{l,q}(\al)$ and $\bC_r(\al)$ are just all plane curve singularities in this case.
\end{theorem}
\begin{proof}
  The case $m=1$ is obvious. So, using the Procedure, we suppose that $m>1$ and $k+r=m$. If $\bB=t^k\bC_{l,q}(\al)+\bD$, a basis of $\hB$ consists of the classes of $\{1,t^{k+l}(e+t^q\al),t^m\tau\}$. Since $t^{k+l}(e+t^q\al)\notin\bJ_m$, the subalgebra $\bS$ must contain the classe of $t^{k+l}(e+t^q\al')$ for some $\al'\in\bD_1$  with $v(\al')=1$. Again the case $k=0$ is impossible. If $k>0$, the preimage of $\bS$ coincides with $t^{k-1}\bC_{l+1,q}+\bD$. If $\bB=t^k\bC_r(\al)+\bD$, the calculations are quite similar.

Now one easily checks that $\ed\bC_{l,q}(\al)=\ed\bC_r(\al)=2$, while $\ed\bC=3$ for all other rings.
\end{proof}

\subsection{Two branches, non-ramified case}
\label{s24}

 We set
\begin{align*}
  \bC_{l,q}(\al)&=\bD+t^l(e_1+t^q\al)\bD+t^r\bA,\text{ where } r=2l+q\\
  &\ \text{ and } \al\in\bD_1\setminus(e_1\bD+t\bD).
\end{align*}
 Then $\al$ can be chosen as $a\th$, where $\th$ is a fixed preimage of $\bar\th$ in $\bD_1$ and $a\in\bD$ is uniquely defined modulo $t^l$. Again $\bC_{0,q}(\al)=\bD+e_1\bD+t^q\bA$ are just all decomposable rings in this case. Especially, $\bC_{0,0}(\al)=\bA$.

\begin{theorem}\label{24}
  Every over-ring of $\bA_m$ coincides with one of the rings $t^k\bC_{l,q}(\al)+\bD$, where $k+r\le m$. The rings $\bC_{l,q}(\al)$ are just all plane curve singularities in this case.
\end{theorem}

We omit the proof in this case, since it practically repeats the calculations in the other cases.

\subsection{Three branches case}
\label{s25}

 We set 
$$
 \bC_{l,q}(\al)=\bD+t^l\al\bD+t^r\bA,
$$
 where $\al=e+t^qae'$, $e\ne e'$ are primitive idempotent in $\bA$, $r=2l+q$, $a\in\bD^\times$ and $a\not\equiv1\!\!\pmod{t}$ if $q=0$. Obviously, $a$ is unique modulo $t^l$. Again $\bC_{0,q}(\al)=\bD+e\bD+t^q\bA$ are just all decomposable rings in this case and $\bC_{0,0}=\bA$. Note also that if $\bC=\bD+t^l\al\bD+t^r\bA$, where $\al=e+ae'$ as above with $a\equiv1\!\!\pmod t$, then, for $a\equiv1\!\!\pmod{t^l}$, $\bC=t^l\bC_{0,q}(1-e-e')+\bD$, and for $a\equiv1\!\!\pmod{t^q}$ with $0<q<l$, $\bC=\bC_{l,q}(\al')$ for some $\al'$.

\begin{theorem}\label{25}
  Every over-ring of $\bA_m$ coincides with $t^k\bC_{l,q}(\al)+\bD$ for some $\al$ and some $l,q$ with $k+r\le m$. The rings $\bC_{l,q}(\al)$ are just all plane curve singularities in this case.
\end{theorem}

We also omit the proof in this case, since it practically repeats the calculations in the other cases.

\subsection{Table of plane curve cubic singularities}
\label{s26}

 We present in Table 1 below all plane curve cubic singularities. In this table $s$ is the number of branches, $^*$ marks the unramified cases (related to the residue field extensions, hence impossible if $\fK$ is algebraically closed); $x,y$ are generators of the maximal ideal, $v(a)$ denotes the \emph{multivaluation} of an element $a\in\bA$, i.e. the vector of valuations of its components with respect to the decomposition of $\bA$ into the product of discrete valuation rings. The column ``type'' shows the correspondence with the Arnold's classification \cite[\S\,15]{avg}. If $\chr\fK=0$ and $\bA$ is ramified, it actually shows the place of the rings in this classification. If $\chr\fK=0$ and $\bC$ is non-ramified, it shows the place of the ring in this classification after the natural extension of the field $\fK$. The validation of this column is given in \cite[Section 2.3]{dg}. Note that, following \cite{dg}, we denote by $E_{l,q}$ the singularities $J_{l,q}$ in the sense of \cite{avg}. Such notations seem more uniform. Note also that the singularities of types $E_1$ and $E_2$ are actually not cubic, but quadratic, and coincide with those of types $A_1$ and $A_2$ of \cite{avg}. Finally, the last column, ``$\pa$'' shows the number of parameters $p$ from the residue field $\fK$ which define a unique ring of this type. We will consider this value in the last section. It does not coincide with the \emph{modality} in the sense of \cite{avg}; the latter equals $p-1$. 

\begin{table}[t]
\caption{}{} \vskip-3ex
 \[
\begin{array}{|c|c|c|c|c|c|}
\hline
 &&&&&\\
s  & \text{name} & v(x) & v(y) & \text{type} & \, \pa \, \\
 &&&&&\\
\hline
 &&&&&\\
1 &  \bC_{2r}(\al) & (3) & (3r+1) & E_{6r} & r\\
  &  \bC_{2r+1}(\al) & (3) & (3r+2) & E_{6r+2} & r\\
 &&&&&\\
\hline
 &&&&&\\
1^* & \bC_r(\al) & (1) & (r) & E^*_{r,0} & r\\ 
 &&&&&\\ 
\hline
 &&&&&\\
2 &  \bC_r(\al) & (2,1) & (2r+1,\8) & E_{6r+1} & r\\
  & \bC_{l,q}(\al) & (2,1) & (2l,\8) & E_{l,2q+1} & l \\
 &&&&&\\
\hline
 &&&&&\\
2^* & \bC_{l,q}(\al) & (1,1) & (l,\8) & E^*_{l,2q} & l\\
 &&&&&\\
\hline
 &&&&&\\
3 & \bC_{l,q}(\al) & (1,1,1) & (l,l+q,\8) & E_{l,2q} & l\\
 &&&&&\\
\hline
\end{array}    
\] \vskip-1ex
\end{table}

\begin{erem}
  The \emph{tame} cubic plane curve singularities $T_{3,q}\ (q\ge6)$ \cite{d1,dg2} coincide with those of types $E_{2,q-6}$.
\end{erem}

\section{Ideals}
\label{s3}

 As we have mentioned above, every ideal of a cubic ring $\bC$ is isomorphic either to an over-ring $\bB\spe\bC$ or to its dual $\bB^*=\hom_\bD(\bB,\bD)$. If $\bC$ is \emph{Gorenstein} (for instance, if it is a plane cubic singularity), then $\bC^*\simeq\bC$, thus $\bB^*\simeq\hom_\bC(\bB,\bC)$. Therefore, to calculate $\bB^*$, one has to choose a Gorenstein subring $\bC\sbe\bB$ and to calculate
\[
 \hom_\bC(\bB,\bC)\simeq\{\la\in\bL\,|\,\la\bB\sbe\bC\}=\{\la\in\bC\,|\,\la\bB\sbe\bC\} 
\]
 (the latter equality holds since $1\in\bB$). This remark easily leads to the following result.

\begin{theorem}\label{dual}
 The duals to the cubic rings are as follows:
  \begin{description}
\item[One branch ramified case] If $\,\bB=\bD+t^k\bC_r(\al)$, then $\bB^*\simeq \bD+t^{[r/2]}\al\bD+t^{k+r}\bA$.

\item[One branch non-ramified case] If $\,\bB=\bD+t^k\bC_r(\al)$, then $\bB^*\simeq\bD+t^r\al\bD+t^{k+2r}\bA$.

\item[Two branches ramified case]
     \begin{enumerate}
     \item If $\,\bB=\bD+t^k\bC_{l,q}(\al)$, then $\bB^*\simeq\bD+t^l(e+t^q\al)\bD+t^{k+2l+q}\bA$.
     \item   If $\,\bB=\bD+t^k\bC_r(\al)$, then $\bB^*\simeq\bD+t^r\al\bD+t^{k+2r+1}\bA$.
     \end{enumerate}

\item[Two branches non-ramified case]  If $\,\bB=\bD+t^k\bC_{l,q}(\al)$, then $\bB^*\simeq\bD+t^l(e+t^q\al)\bD+t^{k+2l+q}\bA$.

\item[Three branches case]  If $\,\bB=\bD+t^k\bC_{l,q}(\al)$, then $\bB^*\simeq\bD+t^l\al\bD+t^{k+2l+q}\bA$.
\end{description}
\end{theorem}
\begin{proof}
  The proof is immediate if we choose for a Gorenstein subring $\bC\sbe\bB$ the plane curve singularity $\bC=\bC_{k+r}(\al)$ or $\bC_{k+l,q}(\al)$ depending on the shape of $\bB$. For instance, in two branches ramified case, if $\bB=\bD+t^k\bC_{l,q}(\al)$ and $\bC=\bD+\bC_{l+k,q}(\al)$, then 
\begin{align*}
\bB^*&\simeq \{\la\in\bC\,|\,\la\bB\sbe\bC\}=t^k\bD+t^{k+l}(e+t^q\al)\bD+t^{2k+2l+q}\bA \\  
  & \simeq\bD+t^l(e+t^q\al)\bD+t^{k+2l+q}\bA.
\end{align*}
\vskip-1.5em
\end{proof}

\begin{corol}\label{gorenstein}
  If a cubic ring is Gorenstein, it is a plane curve singularity.
\end{corol}

 Note that it is no more the case for the extensions of bigger degrees. For instance, the rings $P_{pq}$ from \cite{dg2}, which are quartic, are Gorenstein (they are complete intersections) but of embedding dimension $3$. 

\section{Geometric case. Number of parameters}
\label{s4}

 In this section we suppose that our rings are of \emph{geometric nature}, i.e. $\bD=\fK[[t]]$, where $\fK$ is algebraically closed. Then one can consider the \emph{number of parameters} $\pa(\bC)$ defining $\bC$-ideals (see \cite[Section 2.2]{d1} or \cite[Section 3]{dg1}, where it is denoted by $\pa(1;\bC,\bA)$). Actually, it coincides with the minimal possible number $p$ for which there is a finite set of \emph{families of ideals} $\kI_k\ (1\le k\le m)$ of dimensions at most $p$ such that every $\bC$-ideal is isomorphic to one belonging to some family $\kI_k$. Equivalently, it is the maximal possible $p$ such that is a $p$-dimensional family of ideals $\kI$ where every isomorphism class of ideals only occurs finitely many times. In \cite{dg} a criterion was established in order that $\pa(\bC)\le1$. For cubic rings it means that $\bC$ dominates a singularity of type $E_m\ (18\le m\le 20)$   or $E_{3,i}$. The following results give the exact value of $\pa(\bC)$ for all cubic rings of geometric nature. (Note that no unramified case can occur for such rings.)

\begin{theorem}\label{param}
  If $\bC$ is a cubic ring of geometric nature, $\pa(\bC)\le n$ \iff $\bC$ dominates one of the singularities of type $E_{12n+i}\ (6\le i\le 8)$ or $E_{2n+1,q}\ (q\ge0)$.
\end{theorem}
\begin{proof}
 Certainly, we have to prove that
\begin{enumerate}
\item  every ring of one of the listed types have at most $n$-parameter families of ideals;
\item  if $\bC$ dominates no ring of the listed types, it has $(n+1)$-parameter families of ideals.
\end{enumerate}

 Since the calculations in all cases are similar, we only consider the one branch ramified case. Note first that the rings $\bC_{2r}(\al)$ as well as $\bC_{2r+1}(\al)$ form a $r$-parametric family. Indeed, we can choose in the first case $\al=\tau+a\tau^2$, and in the second one $\al=\tau^2+a\tau^4$, where $a\in\bD$ is defined modulo $t^r$, and such a presentation is unique. The same is true also for $t^k\bC_{2r}(\al)+\bD$ and $t^k\bC_{2r+1}(\al)+\bD$ for any $k$. Since $\bC_{2r}(\al)\spe\bA_{2r}$ for all $\al$, we get $\pa(\bA_{2r})\ge r$.

 Let $\bC$ dominate neither a ring of type $E_{12n+6}$, i.e. $\bC_{4n+2}(\al)$, nor a ring of type $E_{12n+8}$, i.e. $\bC_{4n+3}(\al)$. Then it contains no element of valuation smaller than $6n+6$, so $\bC\sbe\bA_{2n+2}$. Hence, $\pa(\bC)\ge n+1$.

 On the other hand, consider the ring $\bC_{2r+q}(\al)$, where $q\in\{0,1\}$. Its over-rings are of the kind $\bD+t^k\bC_{2m+q}(\be)$, where $k+m\le r$ and $k+2m\le 2r$. Moreover, let $\al=\tau^{q+1}+a\tau^{2q+2}$ and $\be=\tau^{q+1}+b\tau^{2q+2}$. Then $b$ is defined modulo $t^m$ and $b\equiv a\!\!\pmod{t^{r-m-k}}$. Therefore, the over-rings with the fixed $m,k$ form a $p$-parameter family, where $p=\min(m,r-m-k)$. Hence, $2p\le r$ and $p\le[r/2]$. If we set $r=2n+1$, we get that $\pa(\bC_{4n+2}(\al))\le n$ and $\pa(\bC_{4n+3}(\al))\le n$ for all possible $\al$. It accomplishes the proof.
\end{proof}

 Obvious considerations give the number of parameters for special rings.

\begin{corol}\label{cor}
\[
   \begin{split}
 \pa(\bC_r(\al))&=[r/2],\\
 \pa(\bC_{l,q}(\al))&=[l/2],\\
 \pa(\bA_m)&=[m/2].
\end{split} 
\]
\end{corol}

\end{document}